\documentclass[11pt]{amsart}

\usepackage{amsmath,amssymb}
\usepackage{amsthm}
\usepackage[abbrev]{amsrefs}
\usepackage{latexsym}
\usepackage{txfonts}
\usepackage{graphicx}
\usepackage{tikz}
\usepackage{bm}

\allowdisplaybreaks[1]

\title[]{$p$-adic properties of division polynomials and algebraic sigma functions}
\date{}
\author{Yu Katagiri \and Shinichi Kobayashi}

\newtheorem{thm}{Theorem}[section]
\newtheorem*{thm*}{Theorem}
\newtheorem{lem}[thm]{Lemma}
\newtheorem*{lem*}{Lemma}
\newtheorem{prop}[thm]{Proposition}
\newtheorem*{prop*}{Proposition}
\newtheorem{cor}[thm]{Corollary}
\newtheorem*{cor*}{Corollary}

\theoremstyle{definition}
\newtheorem{dfn}[thm]{Definition}
\newtheorem*{dfn*}{Definition}

\newtheorem*{ex*}{Example}
\newtheorem{rmk}[thm]{Remark}
\newtheorem*{rmk*}{Remark}
\newtheorem{conj}[thm]{Conjecture}
\newtheorem*{conj*}{Conjecture}

\newcommand{\Q}{\mathbb{Q}}
\newcommand{\Z}{\mathbb{Z}}
\newcommand{\C}{\mathbb{C}}

\makeatletter
\@addtoreset{equation}{section}

\makeatother

\makeatletter\@namedef{subjclassname@2020}{\textup{2020} Mathematics Subject Classification}\makeatother


\begin{document}

\begin{abstract}
Let $p \geq 5$ be a prime, let $K$ be a finite extension of $\mathbb{Q}_p$,
and let $E/K$ be an elliptic curve with good reduction.
Let $F_n$ denote the $n$-division polynomial of $E$.
Silverman proved that if the reduction is ordinary, then for every
$P \in E(K) \setminus \hat{E}(K)$ and a suitable power $q$ of $p$, the
sequences $(F_{mq^k}(P))_{k \geq 0}$  converge
$p$-adically to limits that are algebraic over the field of definition of $E$.
In both the ordinary and supersingular cases, we show that these sequences
converge, and determine these limits
explicitly in terms of the values of Mumford's algebraic sigma function
attached to the Teichm\"uller lift of the prime-to-$p$ torsion component of the reduction
of $P$.
In particular, the limits are algebraic in the supersingular case as well.
As an application, we obtain explicit $p$-adic limit formulas for
nonsingular elliptic divisibility sequences.
\end{abstract}
\maketitle

\section{Introduction}

Division polynomials are to an elliptic curve what the functions $X^n-1$ are
to the multiplicative group. This analogy is especially suggestive over a
$p$-adic field. If $a$ is a $p$-adic unit and $q$ is a suitable power of $p$,
then $a^{q^k}$ converges to the Teichm\"uller lift of its reduction, and
consequently the sequence $a^{mq^k}-1$ has an explicit $p$-adic limit. This
paper studies an elliptic counterpart of this elementary phenomenon.

Let $F_n$ denote the normalized $n$-division polynomial of an elliptic curve $E$. For a point $P$, the value $F_n(P)$ controls the denominators of the coordinates of $nP$, and the functions $F_n$ satisfy the nonlinear recurrence underlying elliptic divisibility sequences. More precisely, every nonsingular elliptic divisibility sequence is, up to a factor of the form $\gamma^{n^2-1}$, obtained from a sequence of values $F_n(P)$. Thus the $p$-adic behavior of division polynomial values is closely connected with the $p$-adic behavior of elliptic divisibility sequences.

This leads to two distinct questions. First, do suitable $p$-power subsequences $F_{mq^k}(P)$ converge? Second, when they converge, can their limits be described intrinsically and shown to be algebraic? The second question is the main concern of the present paper.

Let $K$ be a field of characteristic different from $2$ and let $E$ be an elliptic curve over $K$ given by the equation
\begin{align}\label{WeiEq}
y^2=4x^3-g_2x-g_3.
\end{align}
We use the local parameter $t=-2x/y \in K(E)$ at the origin $O$. For each positive integer $n$, the {\it $n$-division polynomial} $F_n$ of $E$ is characterized as the unique rational function $F_n \in K(E)$ satisfying
\begin{align*}
{\rm div}(F_n)=[n]^{\ast}(O)-n^2(O), \quad \left(\frac{t^{n^2}F_n}{[n]^{\ast}(t)}\right)(O)=1.
\end{align*}
For a fixed point $P \in E(\overline{K})$, the values $F_n(P)$ encode the denominators of
the coordinates of $nP$. More precisely, there exist rational functions $G_n, H_n \in K(E)$ such that
\begin{align*}
nP=\left(\frac{ G_n(P)}{F_n(P)^2}, \frac{ H_n(P)}{F_n(P)^3}\right)
\end{align*}
for every positive integer $n$ such that $nP \neq O$.
The division polynomials also satisfy the recurrence relation
\begin{align*}
F_{m+n}F_{m-n}=F_{m+1}F_{m-1}F_n^2-F_{n+1}F_{n-1}F_m^2
\end{align*}
for all $1 \leq n < m$. This recurrence relation is the source of the notion
of elliptic divisibility sequences, which we recall in Section~4.  
For the definition of division polynomials for a general Weierstrass
equation over an arbitrary base ring, and for further properties, see \cite[Appendix I]{MT91}.

The starting point of this paper is a theorem of Silverman on the $p$-adic
limits of the sequence $(F_n(P))_{n\geq 1}$. When $K$ is a finite field,
Silverman studied the periodicity of this sequence; when $K$ is a $p$-adic
field, he studied its $p$-adic convergence properties. We recall the latter
result. Let $p$ be an odd prime and let $K$ be a finite extension of $\Q_p$. We write $\mathfrak{m}_K$ for the maximal ideal of the ring of integers $\mathcal{O}_K$ of $K$.

\begin{thm}[{\cite[Theorem 12]{Si05a}}]\label{Silverman thm}
Let $E$ be an elliptic curve over $K$ and let $P \in E(K)$. 
Fix a minimal Weierstrass equation for $E$ and assume that $E$ has good ordinary reduction and that $P$ is not a point of the formal group $\hat{E}(K)$ of $E$.
\begin{enumerate}
\item There exists a power $q=p^N$ such that for every positive integer $m$, the limit
\begin{align*}
\lim_{k\to \infty} F_{mq^k} (P)
\end{align*}
exists and belongs to $\mathcal{O}_K$.
\item The limit in $(1)$ is algebraic over the field generated by the coefficients of the chosen Weierstrass
equation for $E$. In particular, if $E$ is defined over a number field, then the limit is algebraic.
\item Let $r \geq 2$ be the order of the reduction of $P$ modulo the maximal ideal of $K$, and let $r'$ be the prime-to-$p$ part of $r$. Then the limit in $(1)$ is zero if and only if ${r'} \mid m$.
\end{enumerate}
\end{thm}

\begin{rmk}
(1) Under the change of variables $X=x$ and $Y=y/2$, equation \eqref{WeiEq} becomes 
\begin{align*}
Y^2=X^3-\frac{g_2}{4}X-\frac{g_3}{4}
\end{align*}
and Silverman's local parameter $-X/Y$ agrees with our parameter $t=-2x/y$. Hence our normalized division polynomials agree with those of \cite{Si05a} under this isomorphism.

(2) Silverman also gave an explicit choice of the power $q$ in terms of the period of the sequence $(F_n(P) \bmod \mathfrak{m}_K)_{n \geq 1}$.

(3) If $P$ lies in the formal group $\hat{E}(K)$ of $E$, \cite[Theorem 6.1]{St16} implies that the $p$-adic valuations of $F_n(P)$ are not bounded below as $n$ ranges over positive integers. 
Equivalently, the sequence $(F_n(P))_{n \geq 1}$ is unbounded in $K$ for the $p$-adic norm, and hence it does not converge in $K$.
\end{rmk}

The ordinary reduction hypothesis in Theorem~\ref{Silverman thm} comes from the
use of the Mazur--Tate $p$-adic sigma function. 
However, as Silverman explained in his addendum \cite{Si05b}, Ayad's
congruence results for elliptic divisibility sequences imply the $p$-adic
convergence of the corresponding subsequences for all but finitely many
primes.
Such congruence methods, however, yield convergence alone and do not identify the limits.
 The point of the present paper is different. Under the assumption of good reduction, we identify the limiting values explicitly and uniformly in the ordinary and supersingular cases, in terms of an algebraic sigma function arising from Mumford's theory of algebraic theta functions.

We define $V(E)$ to be the set of sequences $\widetilde{P}=(P_n)_{n \geq 1}$ in $E(\overline{K})_{\rm tor}$ satisfying $mP_{mn}=P_n$ for any positive integers $m, n$. 
The set $V(E)$ is an arithmetic analogue of the universal covering space of a complex elliptic curve, and the algebraic sigma function $\sigma^{\rm alg}$ is defined on $V(E)$. (See Definition \ref{def algsigma} and Theorem \ref{an desc} for its analytic description.) For $a \in \Z$ and $\widetilde{P}=(P_n)_{n \geq 1} \in V(E)$, we write $a \widetilde{P}=(aP_n)_{n \geq 1}$.

\begin{thm}\label{main1}
Let $p \geq 5$ be a prime, let $K$ be a finite extension of $\mathbb{Q}_p$,
and let $E$ be an elliptic curve over $K$ with good reduction, given by a
smooth Weierstrass equation \eqref{WeiEq}, that is, with
$g_2, g_3 \in \mathcal{O}_K$ and unit discriminant.
Let $P \in E(K) \setminus \hat{E}(K)$ and let $r \geq 2$ be the order of the
reduction of $P$ modulo the maximal ideal of $K$. Write 
\begin{align*}
r=p^jr', \qquad p \nmid r'
\end{align*}
and set $P'=p^jP$. Let $T \in E(K)[r']$ be the Teichm\"{u}ller lift of the reduction of $P'$, that is, the unique prime-to-$p$ torsion point reducing to $P'$ modulo the maximal ideal. 
Let $\widetilde T=(T_n)_{n \geq 1} \in V(E)$ with $T_1=T$. 
Then there exists a power $q=p^N$ such that for every positive integer $m$, the limit 
\begin{align*}
\lim_{k\to \infty} F_{mq^k} (P)
\end{align*}
exists in $\mathcal{O}_K$ and is given by
\begin{align*}
\lim_{k\to \infty} F_{mq^k} (P)=
\begin{cases}
\displaystyle{\frac{\sigma^{\rm alg}(m\widetilde{T})}{\omega(\sigma^{\rm alg}(\widetilde{T}))^{m^2}}}, 
& \text{if $r$ is prime to $p$}, \\
\displaystyle{0}, & \text{if $r$ is a power of $p$}, \\
\displaystyle{\sigma^{\rm alg}(m'\widetilde{T})\omega\left(\frac{F_{p^j}(P)}{\sigma^{\rm alg}(\widetilde{T})}\right)^{m'^2}}, & \text{otherwise}.
\end{cases}
\end{align*}
Here, $\omega=\omega_L : \mathcal{O}_L^{\times} \rightarrow \mu(L)$ denotes the Teichm\"{u}ller character, where $L=K(E[4r'^2])$, and in the third case, 
\begin{align*}
m'=\frac{mq^j}{p^j}=mp^{(N-1)j}.
\end{align*}
%
%
%
%
\end{thm}

The limit does not depend on the choice of $\widetilde{T}$. (See Remark \ref{depends only on T}.) 
Moreover, the formula holds for any positive integer $N$ satisfying the following two conditions:
\begin{itemize}
\item $p^N \equiv 1 \bmod 2r'^2h$, where $h$ is the order of the multiplicative group of the residue field of $K$,
\item the residue degree of the extension $K(E[4r'^2])/\mathbb{Q}_p$ divides $N$.
\end{itemize}
(Indeed, the proof of Theorem \ref{main1'} uses only these two properties of $N$.)

The construction of algebraic sigma functions also gives analogues of
properties (2) and (3) in Theorem~\ref{Silverman thm} as follows.

\begin{cor}\label{cor alg}
Under the hypotheses and notation of Theorem \ref{main1}, the limit 
\begin{align*}
\lim_{k\to \infty} F_{mq^k} (P)
\end{align*}
is algebraic over the field generated by the coefficients of the chosen Weierstrass equation. 
Moreover, the limit is zero if and only if $r' \mid m$.
\end{cor}

The mechanism behind the formula may be described as follows. 
After separating the $p$-primary part of the reduction order, the relevant point decomposes into a formal group component and a prime-to-$p$ torsion point $T$, the Teichm\"uller lift of its reduction. 
Repeated multiplication by $q$ sends the formal group component to the origin while fixing $T$. 
Mumford's algebraic sigma function records the surviving torsion point $T$,
while the exponent $n^2$ in the normalization of $F_n$ gives rise to the
Teichm\"uller character $\omega$.

The paper is organized as follows. In Section 2, we review Mumford's theory of algebraic sigma functions for elliptic curves and relate it to certain holomorphic functions on $\C$.
In Section 3, we study the $p$-adic properties of algebraic sigma functions and prove Theorem \ref{main1} by relating them to division polynomials. 
In Section 4, we apply Theorem \ref{main1} to obtain explicit $p$-adic limit formulas for nonsingular elliptic divisibility sequences and explain their relation to Silverman's conjecture.

\vspace{10pt}
\noindent
\textsc{Notation:}~ Let $K$ be a field of characteristic zero and $E$ an
elliptic curve over $K$ given by a Weierstrass equation $(\ref{WeiEq})$.
We write $\hat{E}$ for the formal group of $E$ with respect to the local
parameter $t=-2x/y$, and we denote its formal logarithm by $\lambda(t)$.
When $K=\mathbb{C}$, we may write $E(\mathbb{C})=\mathbb{C}/\Lambda$ for a
lattice $\Lambda$ and regard a rational function on $E$ as an elliptic
function of the variable $z \in \mathbb{C}$; its Laurent expansion at
$z=0$ then turns into its formal Laurent expansion at $O$ in the parameter
$t$ under the substitution $z=\lambda(t)$.
In the same way, for a meromorphic function $f(z)$ on $\mathbb{C}$, we
write $\hat{f}(t)$ for the composition of its Laurent expansion at the
origin with $z=\lambda(t)$; likewise, for a formal Laurent series
$f(z) \in K((z))$, we put
\begin{align*}
\hat{f}(t) \coloneqq f(\lambda(t)) \in K((t)).
\end{align*}%

\vspace{10pt}
\noindent
\textsc{Acknowledgments:}~ The authors are grateful to Takao Yamazaki and Kazuto Ota for many helpful discussions and comments. The authors also thank Satoshi Kumabe for his careful reading of the manuscript.
The first author was supported by JSPS Grant-in-Aid for Transformative Research Areas (A) (22H05107). The second author was supported by JSPS KAKENHI Grant Number 22H00096.

\section{Mumford's algebraic theta functions}

Mumford introduced and studied algebraic theta functions on abelian varieties.
We review Mumford's theory for elliptic curves. This section follows \cite{BK10}; see also \cite{Mu91}.

In this section, we let $K$ be a field of characteristic $0$ and fix an algebraic closure $\overline{K}$ of $K$. Let $E$ be an elliptic curve over $K$.

\subsection{Definition}
For a prime $\ell$, let $T_\ell (E)$ be the $\ell$-adic Tate module of $E$. We define
\begin{align*}
V(E) &\coloneqq \mathbb{Q} \otimes_{\mathbb{Z}} \prod_{\ell} T_{\ell}(E) 
=\left\{(P_n)_n \in \prod_{n \geq 1} E(\overline{K})_{\rm tor} \ \middle| \ mP_{mn}=P_n \text{ for all } m, n \geq 1 \right\},
\end{align*}
where the product runs over all rational primes $\ell$.

\begin{lem}\label{principal lem}
Let $n$ be a positive integer and $P \in E(\overline{K})[n]$. Then the divisor 
\begin{align*}
[n]^{\ast}\left(\tau_P^{\ast} (O)-(O)\right)=\sum_{nQ=-P}(Q)-\sum_{nR=O}(R)
\end{align*}
is principal.
\end{lem}

\begin{proof}
Take a point $Q_0 \in E(\overline{K})$ such that $nQ_0=-P$. As a sum of points of $E$, we have
\begin{align*}
\sum_{nQ=-P}Q-\sum_{nR=O}R=\sum_{nR=O}(Q_0+R)-\sum_{nR=O}R=n^2Q_0=O.
\end{align*}
Thus the lemma follows from Abel's theorem.
\end{proof}

Let $\tilde{P}=(P_n)_{n \geq 1} \in V(E)$ and write $Q_n \coloneqq P_{2n}$ for $n \geq 1$. Let $N$ be the order of $P_1$. By Lemma \ref{principal lem}, we can take a rational function $g \in \overline{K}(E)$ whose divisor is ${\rm div}(g)=[2N]^{\ast}(\tau_{Q_1}^{\ast} (O)-(O))$. We put
\begin{align*}
f_{\tilde{P}} \coloneqq \tau_{Q_{2N}}^{\ast}(g/[-1]^{\ast}g) \in \overline{K}(E).
\end{align*}
Note that $f_{\tilde{P}}$ does not depend on the choice of $g$. We also see that ${\rm div}(f_{\tilde{P}})=[2N]^{\ast}(\tau_{P_1}^{\ast} (O)-(O))$.

\begin{rmk}\label{f as ell func}
We consider the case $K=\mathbb{C}$. We take a lattice $\Lambda$ in $\mathbb{C}$ such that $\mathbb{C}/\Lambda \simeq E(\mathbb{C})$. If we regard rational functions $f_{\tilde{P}}, g \in \mathbb{C}(E)$ as elliptic functions with respect to the period lattice $\Lambda$ via this isomorphism, then we have
\begin{align*}
f_{\tilde{P}}(z)=\frac{g(z+w_{2N})}{g(-z-w_{2N})},
\end{align*}
where $w_{2N} \in \mathbb{C}$ corresponds to the point $Q_{2N} \in E(\mathbb{C})$.
\end{rmk}

\begin{dfn}\label{def algsigma}
Let $t$ be a local parameter at $O$. For $\tilde{P}=(P_n)_{n \geq 1} \in V(E)$ with $P_1 \in E_{\rm tor}$ of order $N$, we define 
\begin{align*}
\sigma^{\rm alg}(\tilde{P}) \coloneqq \left(f_{\tilde{P}} \cdot [2N]^{\ast}(t)\right)(O).
\end{align*}
This gives the {\it algebraic sigma function} $\sigma^{\rm alg} : V(E) \rightarrow \overline{K}$ for $E$.
\end{dfn}

\begin{rmk}\label{rmk algsigma}
(1) The value $\sigma^{\rm alg}(\tilde{P})$ depends on the choice of a local parameter $t$. In the following, we use the local parameter $t=-2x/y$ with the Weierstrass equation \eqref{WeiEq}.

(2) Although the algebraic sigma function is defined on $V(E)$, the value $\sigma^{\rm alg}(\tilde{P})$ for $\tilde{P}=(P_n)_{n \geq 1}$ is determined by the single point $P_{4N}$.

(3) The rational function $f_{\tilde{P}}$ has order $-1$ at $O$ if $P_1 \neq O$ and order 0 at $O$ if $P_1=O$. Hence $\sigma^{\rm alg}(\tilde{P}) \neq 0$ if and only if $P_1 \neq O$.
\end{rmk}


The following proposition follows from the construction of the algebraic sigma function.

\begin{prop}\label{definition field}
For $\tilde{P}=(P_n)_{n \geq 1} \in V(E)$ with $P_1 \in E_{\rm tor}$ of order $N$, we have $\sigma^{\rm alg}(\tilde{P}) \in K(E[4N^2])$.
\end{prop}

\subsection{The analytic description of $\sigma^{\rm alg}$}

Let $K$ be a field of characteristic $0$, $E$ an elliptic curve over $K$, and $\omega_E$ an invariant differential form on $E$. Given an embedding $K \hookrightarrow \C$, there is a unique lattice $\Lambda \subset \C$ and a uniformization $\phi : \C/\Lambda \rightarrow E(\C)$ such that $\phi^{\ast}\omega_E=dz$; with this normalization, $E$ is represented by equation \eqref{WeiEq}.
We recall below several functions associated with the lattice  $\Lambda$. See {\cite[VI \S 3]{AEC}} and {\cite[I \S 5]{ATAEC}} for more detailed properties.

We define the {\it Weierstrass sigma function} (associated to $\Lambda$)
\begin{align}\label{Wsigma}
\sigma(z)=\sigma(z ; \Lambda) \coloneqq z \prod_{\gamma \in \Lambda \setminus \{0\}} \left(1-\frac{z}{\gamma}\right)\exp \left(\frac{z}{\gamma}+\frac{1}{2}\left(\frac{z}{\gamma}\right)^2\right).
\end{align}
It is known that the infinite product in (\ref{Wsigma}) is absolutely and uniformly convergent on compact subsets of $\C$ and defines a holomorphic function on $\mathbb{C}$ with simple zeros on $\Lambda$ and no other zeros. If $E$ is defined by the Weierstrass equation (\ref{WeiEq}), it is also known that the Taylor expansion of $\sigma(z)$ at $z=0$ is given by
\begin{align}\label{exp sigma}
\sigma(z)=\sum_{m, n=0}^\infty a_{m, n} \left(\frac{g_2}{2}\right)^m (2g_3)^n \frac{z^{4m+6n+1}}{(4m+6n+1)!},
\end{align}
where $a_{m, n}$ is defined by the recurrence relation
\begin{align*}
a_{m, n}=&3(m+1)a_{m+1, n-1} \\
&+\frac{16}{3}(n+1)a_{m-2, n+1}-\frac{1}{3}(2m+3n-1)(4m+6n-1)a_{m-1, n}
\end{align*}
with $a_{0, 0}=1$ and $a_{m, n}=0$ if $m$ or $n$ is negative \cite[pp. 635--636]{AS64}. 
We also define the Weierstrass zeta function (associated to $\Lambda$) $\zeta(z)\coloneqq \sigma'(z)/\sigma(z).$ For $\gamma \in \Lambda$ and $z \in \mathbb{C} \setminus \Lambda$, it is known that $\eta(\gamma) \coloneqq \zeta(z+\gamma)-\zeta(z)$ depends only on $\gamma$. This defines a group homomorphism $\eta : \Lambda \rightarrow \mathbb{C}$, which is called the quasi-period map associated to $\Lambda$. We extend the homomorphism $\eta : \Lambda \rightarrow \mathbb{C}$ to $\Lambda \otimes \mathbb{Q}$ linearly.
The Weierstrass sigma function satisfies the transformation formula
\begin{align}\label{transW}
\sigma(z+\gamma)=\psi(\gamma)\exp \left(\eta(\gamma)\left(z+\frac{1}{2}\gamma\right)\right)\sigma(z)
\end{align}
for $z \in \mathbb{C}$ and $\gamma \in \Lambda$, where $\psi$ is a function on $\Lambda$ with values in $\{ \pm 1\}$, and $\psi(\gamma)=1$ if and only if $\gamma \in 2\Lambda$.

\begin{dfn}\label{sigma_z}
For $z_1 \in \Lambda \otimes \mathbb{Q}$, we define
\begin{align*}
\sigma_{z_1}(z) \coloneqq \exp \left(-\eta(z_1)\left(z+\frac{1}{2}z_1\right)\right)\sigma(z+z_1).
\end{align*}
\end{dfn}

Let $z_1 \in \Lambda \otimes \mathbb{Q}$, and let $N$ be the order of its image in 
$\mathbb{C}/\Lambda$.
If we let $P_n \in E(\mathbb{C})$ be the point corresponding to $z_1/n \in \Lambda \otimes \mathbb{Q}$, then $\tilde{P} \coloneqq (P_n)_n \in V(E)$. Thus we have a natural inclusion $\Lambda \otimes \mathbb{Q} \hookrightarrow V(E)$. 
The following theorem gives an analytic description of $\sigma^{\rm alg}$ in terms of the holomorphic function $\sigma_{z_1}$.

\begin{thm}[cf. {\cite[Theorem 2.13]{BK10}}]\label{an desc}
Let $\sigma^{\rm alg}$ be the algebraic sigma function defined in Definition \ref{def algsigma} with respect to the local parameter $t=-2x/y$ at $O$.
Under the above inclusion, we have
\begin{align*}
\sigma^{\rm alg}(\tilde{P})=\sigma_{z_1}(0)=\exp \left(-\frac{1}{2}z_1\eta(z_1)\right)\sigma(z_1).
\end{align*}
In particular, we have $\sigma_{z_1}(0) \in \mathbb{Q}(g_2, g_3, E[4N^2])$. 
Furthermore, the Taylor coefficients of $\sigma_{z_1}(z)$ at $z=0$ are contained in $\mathbb{Q}(g_2, g_3, E[4N^2])$.
\end{thm}

To prove Theorem \ref{an desc}, we prepare several lemmas. 
For $z_1, w_1 \in \Lambda \otimes \mathbb{Q}$, we define
\begin{align*}
\langle z_1, w_1 \rangle \coloneqq \exp \left(z_1 \eta(w_1)-w_1 \eta(z_1)\right).
\end{align*}

\begin{lem}\label{rootlem}
Let $z_1, w_1 \in \Lambda \otimes \mathbb{Q}$ and let $M, N$ be positive integers. If $Mz_1, Nw_1 \in \Lambda$, then $\langle z_1, w_1 \rangle$ is an $MN$-th root of unity.
\end{lem}

\begin{proof}
The Legendre relation \cite[I Proposition 5.2(d)]{ATAEC} implies that $Mz_1 \eta(Nw_1)-Nw_1 \eta(Mz_1) \in 2\pi i \mathbb{Z}$. Hence we have $\langle z_1, w_1 \rangle^{MN}=1$.
\end{proof}

\begin{lem}\label{transWt}
Let $z_1 \in \Lambda \otimes \mathbb{Q}$ and $\gamma \in \Lambda$.
We have
\begin{align*}
\sigma_{z_1}(z+\gamma)=\psi(\gamma) \langle z_1, \gamma \rangle \exp \left(\eta(\gamma)\left(z+\frac{1}{2}\gamma\right)\right)\sigma_{z_1}(z).
\end{align*}
\end{lem}

\begin{proof}
This follows from the transformation formula $(\ref{transW})$ and the definition of $\langle z_1, \gamma \rangle$.
\end{proof}

\begin{rmk}
Viewed as a function of $z_1$, $\sigma_{z_1}(z)$ is not periodic with respect to $\Lambda$.
More precisely, we have 
\begin{align}\label{transWt'}
\sigma_{z_1+\gamma}(z)=\psi(\gamma) \langle \frac{1}{2}z_1, \gamma \rangle \sigma_{z_1}(z)
\end{align}
for all $\gamma \in \Lambda$. In particular, if $z_1 \in 2\Lambda$, then we have $\sigma_{z_1}(z)=\sigma(z)$.
\end{rmk}

Combining the transformation formula $(\ref{transW})$ with Lemma \ref{rootlem} and Lemma \ref{transWt}, we obtain the following proposition.

\begin{prop}\label{ell func lem}
Let $z_1 \in \Lambda \otimes \mathbb{Q}$ and let $N$ be a positive integer. If $Nz_1 \in \Lambda$, then $\sigma_{z_1}(Nz)/\sigma(Nz)$ is periodic with respect to $\Lambda$.
\end{prop}

Proposition \ref{ell func lem} allows us to view $\sigma_{z_1}(Nz)/\sigma(Nz)$ as a rational function on $E$. Note that we identify $z_1 \in \Lambda \otimes \mathbb{Q}$ with the point $P_1 \in E(\mathbb{C})$. Since the Weierstrass sigma function $\sigma(z)$ is a holomorphic function with simple zeros on $\Lambda$ and no other zeros, we find that
\begin{align*}
{\rm div}\left(\sigma_{z_1}(Nz)/\sigma(Nz)\right)=[N]^{\ast}(\tau_{P_1}^{\ast} (O)-(O)).
\end{align*}
(This also gives an explicit proof of Lemma \ref{principal lem} in the case $K=\mathbb{C}$.)

\begin{proof}[Proof of Theorem \ref{an desc}]
If we let $P_n \in E(\mathbb{C})$ be the point corresponding to $z_n \coloneqq z_1/n \in \Lambda \otimes \mathbb{Q}$, then $\tilde{P} \coloneqq (P_n)_n \in V(E)$. We also put $w_n=z_{2n}$.

For simplicity, put $F=\mathbb{Q}(g_2, g_3, E[4N^2])$. 
We take a rational function $g \in F(E)$ whose divisor is 
\[{\rm div}(g)=[2N]^{\ast}(\tau_{Q_1}^{\ast} (O)-(O)).\] If we identify rational functions on $E$ with elliptic functions with the period lattice $\Lambda$, we may assume that $g(z)=c\sigma_{w_1}(2Nz)/\sigma(2Nz)$ 
for some constant $c \in \mathbb{C}$ by Proposition \ref{ell func lem}. Then, using Remark \ref{f as ell func} and the fact that $\sigma$ is an odd function, we have
\begin{align*}
f_{\tilde{P}}(z)&=\frac{\sigma_{w_1}(2Nz+w_1)}{\sigma(2Nz+w_1)} \cdot \frac{\sigma(-2Nz-w_1)}{\sigma_{w_1}(-2Nz-w_1)} \\
&=\exp \left(-\eta(z_1)\left(2Nz+\frac{1}{2}z_1\right)\right) \frac{\sigma(2Nz+z_1)}{\sigma(2Nz)}
=\frac{\sigma_{z_1}(2Nz)}{\sigma(2Nz)}
\end{align*}
as an elliptic function. Hence we obtain
\begin{align*}
\sigma^{\rm alg}(\tilde{P})&=f_{\tilde{P}}(\lambda(t))\cdot [2N]^\ast(t) |_{t=0}\\
&=\sigma_{z_1}(\lambda([2N]^\ast(t)))\cdot \frac{[2N]^\ast(t)}{\sigma(\lambda([2N]^\ast(t)))} \mid_{t=0}=\sigma_{z_1}(0).
\end{align*}
The second assertion follows immediately from the first assertion and Proposition~\ref{definition field}.
By construction, $f_{\tilde{P}}(z)$ is a rational function over $F$.
Hence, the Laurent coefficients of $f_{\tilde{P}}(z)$ and $\sigma(2Nz)$ at $z=0$ are in $F$. Since $\sigma_{z_1}(2Nz)=f_{\tilde{P}}(z)\sigma(2Nz)$, the last assertion follows.
 \end{proof}

\section{The algebraic sigma series and $p$-adic limits of division polynomials}

In this section, we associate to each $\tilde{T} \in V(E)$ its
{\it algebraic sigma series} $\widehat{\sigma_{\tilde{T}}}(t)$, a formal
power series whose constant term is $\sigma^{\rm alg}(\tilde{T})$, and
prove Theorem \ref{main1}.
We keep the notation of Section 1; in particular, $E$ is an elliptic curve
given by a Weierstrass equation \eqref{WeiEq} over a field $K$ of
characteristic $0$, and $t=-2x/y$ is the local parameter at $O$.

Let $\tilde{T}=(T_n)_{n \geq 1} \in V(E)$ with $T_1 \in E_{\rm tor}$ of
order $N \geq 2$. By Lemma \ref{principal lem}, there exists a unique
rational function $g_{\tilde{T}} \in \overline{K}(E)$ whose divisor is
\[{\rm div}(g_{\tilde{T}})=[N]^{\ast}(\tau_{T_1}^{\ast} (O)-(O))\] and which
satisfies $(tg_{\tilde{T}}(t))(O)=N^{-1}\sigma^{\rm alg}(\tilde{T})$.
We note that $g_{\tilde{T}}$ is defined over $K(E[4N^2])$
(cf.\ Proposition \ref{definition field}), and we regard
$g_{\tilde{T}}(t) \in K(E[4N^2])((t))$ as a formal Laurent series via the
formal completion at $O$.

Since the sigma function is not a rational function on $E$, 
its \(t\)-expansion is not obtained by completing a rational function at \(O\). We instead define it by
$\hat{\sigma}(t)=\sigma(\lambda(t))$, by the convention of Section 1
applied to the formal power series \eqref{exp sigma}.
Over $\mathbb{C}$, if $z_1 \in \Lambda \otimes \mathbb{Q}$ corresponds to
$\tilde{T}$ as in Remark \ref{ASSoverC}, then
$\sigma_{z_1}(z)=\sigma(z)\,g_{\tilde{T}}(z/N)$, which motivates the
following definition.

\begin{dfn}\label{def ASS}
(1) For $\tilde{T}=(T_n)_{n \geq 1} \in V(E)$ with $T_1 \in E_{\rm tor}$ of
order $N \geq 2$, we define the {\it algebraic sigma series}
$\widehat{\sigma_{\tilde{T}}}(t)$ to be
\begin{align*}
\widehat{\sigma_{\tilde{T}}}(t) \coloneqq \sigma(\lambda(t))\,
g_{\tilde{T}}([N^{-1}]^{\ast}(t)),
\end{align*}
where $\sigma(z)$ is the formal power series given by \eqref{exp sigma},
and $[N^{-1}]^{\ast}(t) \in K[[t]]$ denotes the compositional inverse of
the power series $[N]^{\ast}(t)=Nt+O(t^2)$, which exists since
${\rm char}(K)=0$.
A priori the right-hand side lies in $K(E[4N^2])((t))$, but the simple pole
of $g_{\tilde{T}}([N^{-1}]^{\ast}(t))$ at $t=0$ cancels against the simple
zero of $\sigma(\lambda(t))$, so that
$\widehat{\sigma_{\tilde{T}}}(t) \in K(E[4N^2])[[t]]$.

(2) If $\tilde{T}=(T_n)_{n \geq 1} \in V(E)$ with $T_1=O$, we set $N=1$ and
define $\widehat{\sigma_{\tilde{T}}}(t) \coloneqq \psi(\tilde{T})
\sigma(\lambda(t))$, where $\psi(\tilde{T}) \in \{\pm 1\}$ and
$\psi(\tilde{T})=1$ if and only if $T_2=O$.
\end{dfn}

By construction, we have
$\widehat{\sigma_{\tilde{T}}}(0)=\sigma^{\rm alg}(\tilde{T})$, since
$[N^{-1}]^{\ast}(t)=N^{-1}t+O(t^2)$ and
$(tg_{\tilde{T}}(t))(O)=N^{-1}\sigma^{\rm alg}(\tilde{T})$.

\begin{rmk}\label{ASSoverC}
We set $K=\mathbb{C}$ and fix a uniformization
$\mathbb{C}/\Lambda \simeq E(\mathbb{C})$. If we take
$z_1 \in \Lambda \otimes \mathbb{Q}$, under the inclusion
$\Lambda \otimes \mathbb{Q} \hookrightarrow V(E)$, we find that
$\sigma_{z_1}(\lambda(t))=\widehat{\sigma_{\tilde{P}}}(t)$. Here,
$\tilde{P}=(P_n)_{n \geq 1} \in V(E)$ and $P_n$ corresponds to
$z_1/n \in \Lambda \otimes \mathbb{Q}$ for each $n$.
Indeed, if $P_1$ has order $N \geq 2$, the characterization of
$g_{\tilde{P}}$ by its divisor and its normalization at $O$ gives
$g_{\tilde{P}}(z)=\sigma_{z_1}(Nz)/\sigma(Nz)$ as elliptic functions;
since $\lambda([N^{-1}]^{\ast}(t))=\lambda(t)/N$, the substitution
$t \mapsto [N^{-1}]^{\ast}(t)$ realizes $z \mapsto z/N$, and hence
$\sigma(\lambda(t))\,g_{\tilde{P}}([N^{-1}]^{\ast}(t))
=\sigma_{z_1}(\lambda(t))$.
The case $P_1=O$ follows from \eqref{transWt'}.
\end{rmk}

\subsection{$p$-adic properties of sigma functions}

For the rest of this section, let $p \geq  5$ be a prime, let $K$ be a finite extension of $\Q_p$, and
assume that $E$ is given by a smooth Weierstrass equation of the form $y^2 = 4x^3-g_2x-g_3$
(in particular, $g_2, g_3 \in \mathcal{O}_K$).

The following lemma, which is the key lemma in the proof of Theorem \ref{main1}, gives $p$-adic properties of algebraic sigma functions $\sigma^{\rm alg}$.

\begin{lem}\label{keylem}
Let $\tilde{T}=(T_n)_{n \geq 1} \in V(E)$ with $T_1 \in E_{\rm tor}$ of order $N$. Suppose that $N \geq 2$ and that $p \nmid N$.
\begin{enumerate}
\item 
The value $\sigma^{\rm alg}(\tilde{T})$ is a $p$-adic unit in $K(E[4N^2])$.
\item Let $a$ be an element of the maximal ideal of $\mathcal{O}_K$. If the absolute value of $a$ is sufficiently small, then $\widehat{\sigma_{\tilde{T}}}(a)$ is a $p$-adic unit in $K(E[4N^2])$ and we have $|\widehat{\sigma_{\tilde{T}}}(a)-\sigma^{\rm alg}(\tilde{T})|<1$.
\end{enumerate}
\end{lem}

\begin{proof}
\noindent\textup{(1)} Let $L=K(E[4N^{2}])$, let $R$ be the ring of integers of $L$, and let $\kappa$ be its residue field. We denote the normalized valuation on $L$ by $v_L$. Let $f: \mathcal{E} \to \mathrm{Spec}(R)$ be the smooth projective relative curve defined by the smooth Weierstrass equation \eqref{WeiEq}. For each $P\in E(L)$, let $\mathcal{P}$ denote the corresponding section in $\mathcal{E}(R)$, and let $\overline{P}$ denote its reduction to the special fiber $\mathcal{E}_\kappa = \mathcal{E} \times_R \kappa$.

Since $p$ is an odd prime, $-2 \in R^\times$. The smoothness of the Weierstrass equation implies that the coordinates $(x/y, 1/y)$ give a smooth local chart of $\mathcal{E}$ around the identity section $\mathcal{O}$. Thus, the rational function $t = -2x/y \in L(E)$ restricts to a local uniformizer at the origin $\overline{O} \in \mathcal{E}_\kappa$. Consequently, the Cartier divisor of $t$ on $\mathcal{E}$ is purely horizontal in an open neighborhood of $\overline{O}$.

Consider the divisor
\[
  D=[2N]^*\bigl(\tau_{\mathcal{T}_2}^*(\mathcal{O})-(\mathcal{O})\bigr)
\]
on \(\mathcal E\). As a linear combination of sections, \(D\) is
purely horizontal. By the same group-law computation as in
Lemma \ref{principal lem}, interpreted via Abel's theorem for the smooth relative
curve \(\mathcal E/R\) (cf.~\cite[Theorem~2.1.2]{KM85}), the class of
\(\mathcal O_{\mathcal E}(D)\) is trivial in
\(\operatorname{Pic}_{\mathcal E/R}(R)\).
The exact sequence of Picard groups associated with $f$ implies that $\mathcal{O}_{\mathcal{E}}(D) \cong f^*\mathcal{M}$ for some line bundle $\mathcal{M}$ on $\mathrm{Spec}(R)$. Since $R$ is a discrete valuation ring, $\mathrm{Pic}(R) = 0$, which yields $\mathcal{O}_{\mathcal{E}}(D) \cong \mathcal{O}_{\mathcal{E}}$. Hence, there exists a rational function $g \in L(E)^{\times}$ such that $\mathrm{div}_{\mathcal{E}}(g) = D$. Since $D$ has no vertical components, $g$ is uniquely determined up to $R^{\times}$, and its divisor has zero multiplicity along $\mathcal{E}_\kappa$.

We define the rational function $f_{\tilde{T}} \coloneqq \tau_{\mathcal{T}_{4N}}^{*}(g/[-1]^{*}g)$. Since translations and the involution $[-1]$ are automorphisms of $\mathcal{E}$ over $R$, the divisor of $f_{\tilde{T}}$ is also purely horizontal:
\begin{align}\label{div_E}
\mathrm{div}_{\mathcal{E}}(f_{\tilde{T}}) = \sum_{2N\mathcal{Q}=-\mathcal{T}_{1}} (\mathcal{Q}) - \sum_{2N\mathcal{R}=\mathcal{O}} (\mathcal{R}).
\end{align}
Note that $f_{\tilde{T}} \in L(E)^{\times}$ is well-defined independently of the choice of $g$.

Since $p \nmid 2N$, the multiplication map $[2N^2]: \mathcal{E} \to \mathcal{E}$ is finite \'{e}tale over $\mathrm{Spec}(R)$, which implies that the reduction map modulo $\mathfrak{p}_L$ is injective on $\mathcal{E}[2N^2](R)$. As $T_1$ has order $N \ge 2$, we have $-\overline{T}_1 \neq \overline{O}$. Thus, $\overline{Q} \neq \overline{O}$ for all $\mathcal{Q}$ appearing in the first sum of the right-hand side in \eqref{div_E}. This means that $\mathcal{O}$ is the only section in the support of $\mathrm{div}_{\mathcal{E}}(f_{\tilde{T}})$ passing through $\overline{O}$, where $f_{\tilde{T}}$ has a simple pole.

Consider the rational function $\Phi \coloneqq f_{\tilde{T}} \cdot [2N]^{*}(t)$ on $\mathcal{E}$. The pullback $[2N]^{*}(t)$ is a uniformizer at $\mathcal{O}$ and its divisor has a simple zero along $\mathcal{O}$ without vertical components. Therefore, in the divisor $\mathrm{div}_{\mathcal{E}}(\Phi)$, the simple pole of $f_{\tilde{T}}$ and the simple zero of $[2N]^{*}(t)$ along $\mathcal{O}$ cancel each other. As a result, $\mathrm{div}_{\mathcal{E}}(\Phi)$ contains neither the section $\mathcal{O}$ nor the special fiber $\mathcal{E}_\kappa$ in its support.

Since $\mathcal{E}$ is a regular scheme, a rational function is an invertible element in the local ring $\mathcal{O}_{\mathcal{E}, \overline{O}}$ if and only if its Weil divisor has multiplicity zero along every prime divisor passing through $\overline{O}$. By the preceding arguments, no prime divisor in the support of $\mathrm{div}_{\mathcal{E}}(\Phi)$ passes through $\overline{O}$. Hence, $\Phi$ is invertible in $\mathcal{O}_{\mathcal{E}, \overline{O}}$. Its reduction at $\overline{O}$ is therefore nonzero, so $\Phi(\mathcal{O}) \in R^\times$. Hence $v_L(\sigma^{\rm alg}(\tilde{T}))=v_L(\Phi(\mathcal{O}))=0$. This proves~(1).

\medskip
\noindent\textup{(2)} By construction, we have $\widehat{\sigma_{\tilde{T}}}(0)=\sigma^{\rm alg}(\tilde{T})$. By part (1), this value is a $p$-adic unit in $L$. 
By definition, $\widehat{\sigma_{\tilde{T}}}(t)=\sigma(\lambda(t))\,g_{\tilde{T}}([N^{-1}]^{\ast}(t))$; since $p \nmid N$, the series $[N^{-1}]^{\ast}(t)$ has coefficients in $\mathcal{O}_K$.
The first factor has positive radius of convergence by \cite{BKY17}, and the second is a convergent Laurent series near $O$. Hence $\widehat{\sigma_{\tilde{T}}}(t)$ has positive radius of convergence. 
Thus, if $|a|$ is sufficiently small, we have $|\widehat{\sigma_{\tilde{T}}}(a)-\sigma^{\rm alg}(\tilde{T})|<1$ and $\widehat{\sigma_{\tilde{T}}}(a)$ is a $p$-adic unit.
%
%
\end{proof}

\subsection{Teichm\"{u}ller lifts and residue discs}

In preparation for the proof of Theorem \ref{main1}, we recall the
Teichm\"{u}ller lift of a point on an elliptic curve and establish a
convergence lemma for rational functions on residue discs
(Lemma \ref{pWPT}).


\begin{dfn}[{\cite[Proposition 10]{Si05a}}]
Let $P \in E(K)$ and suppose that the order $r$ of its reduction is prime
to $p$. Then the limit
\begin{align*}
 \lim_{\substack {k \to \infty \\ p^k \equiv 1 \bmod r}} p^kP
\end{align*}
exists and belongs to $E(K)[r]$; we call it the
{\it Teichm\"{u}ller lift} of $P$.
\end{dfn}

\begin{lem}\label{pWPT}
Let $E$ be an elliptic curve over $K$ with good reduction, let
$f \in K(E)$ be a rational function, and let $P \in E(K)$.
Suppose that $f$ has no poles on the residue disc
$V_P$, the set of points of $E(\mathbb{C}_p)$ whose 
reduction equals that of $P$. 

\begin{enumerate}
\item Let $\mathcal{E}/\mathcal{O}_K$ be the smooth model of $E$, and let
$t$ be a local parameter along the section over $\mathcal{O}_K$ defined by
$P$. 
Regarding $f$ as an element of $K((t))$ via the formal completion along
this section, that is, via the embedding $K(E) \hookrightarrow K((t))$,
we have $f \in \mathcal{O}_K[[t]] \otimes K$.
\item Suppose that the order $r$ of the reduction of $P$ is prime to $p$,
and let $T$ be the Teichm\"{u}ller lift of $P$. Then we have
\begin{align*}
\lim_{\substack {k \to \infty \\ p^k \equiv 1 \bmod r}} f(p^kP)=f(T).
\end{align*}
\end{enumerate}
\end{lem}

\begin{proof}
\noindent\textup{(1)} By translating by $P$, we may assume that $P=O$.
We note that $t$ identifies the residue disc $V_O$ with the open unit disc $\{ a \in \C_p \mid |a|<1\}$. 
We may write
\begin{align*}
f(t)=\frac{g(t)}{h(t)}
\end{align*}
for some $g(t), h(t) \in \mathcal{O}_K[[t]]$.
(Indeed, $f$ is a rational function in $x$ and $y$, and $t^2x(t), t^3y(t) \in  \mathcal{O}_K[[t]]$.)
Since $f$ has no pole on the residue disc, by the $p$-adic Weierstrass preparation theorem, 
we may take 
$h(t)\in \mathcal{O}_K[[t]]^{\times}$ up to a factor in $K^{\times}$, and the assertion follows. 

\medskip
\noindent\textup{(2)} Let $t$ be a local parameter along the section defined by $T$ on the smooth model.
Then the residue disc $V_T=V_P$ is identified with $T+V_O$.
Write $P=T+R$ and let $t_R$ be the coordinate of $R$ in $V_O$.
If $p^k \equiv 1 \bmod r$, then $p^kT=T$ since $T \in E(K)[r]$, and hence $p^kP=T+p^kR$. 
The point $p^kR$ corresponds to the coordinate 
$[p^k]t_R \in \mathfrak{m}_K^k$. 
By (1) applied to $T$ in place of $P$, we have  $f(t) \in \mathcal{O}_K[[t]] \otimes K$ and  
\begin{align*}
\lim_{\substack {k \to \infty \\ p^k \equiv 1 \bmod r}} f(p^kP)
=\lim_{\substack {k \to \infty \\ p^k \equiv 1 \bmod r}} f([p^k]t_R)
=f(0)=f(T).
\end{align*}
\end{proof}

\subsection{Proof of Theorem \ref{main1}}

We first prove the following theorem, which treats the case $p \nmid r$. The remaining cases are then deduced from this result and the argument in the proof of \cite[Theorem 12]{Si05a}.

\begin{thm}\label{main1'}
Let $E$ be an elliptic curve over $K$ with good reduction, given by a smooth Weierstrass
equation \eqref{WeiEq}. Let $P \in E(K) \setminus \hat{E}(K)$ and let $r \geq 2$ be the order of the reduction of $P$. 
We suppose that $p \nmid r$ and let $T \in E(K)[r]$ be the Teichm\"{u}ller lift of $P$. 
Put $L_0=K(E[4r^2])$, and let
$\omega=\omega_{L_0}:\mathcal{O}_{L_0}^{\times}\to\mu(L_0)$
be the Teichm\"uller character.
Let $\widetilde{T}=(T_n)_{n \geq 1} \in V(E)$ with $T_1=T$.
Then there exists a power $q=p^N$ such that for every positive integer $m$, the limit
$\lim_{k\to \infty} F_{mq^k}(P)$ exists in $\mathcal{O}_K$ and is given by
\begin{align*}
\lim_{k\to \infty} F_{mq^k} (P)=\frac{\sigma^{\rm alg}(m\widetilde{T})}{\omega(\sigma^{\rm alg}(\widetilde{T}))^{m^2}}.
\end{align*}
\end{thm}

\begin{proof}[Proof of Theorem \ref{main1'}]
Let $h$ be the order of the multiplicative group of the residue field of $K$. 
By \cite[Corollary 9]{Si05a}, the sequence $(F_n(P) \bmod \mathfrak{m}_K )_{n \geq 1}$ is purely periodic and has a period dividing $rh$ if $r \geq 3$ and $2rh$ if $r=2$. 
We let $N \in \mathbb{Z}_{>0}$ be the product of $\varphi (2r^2h)$ and the residue degree of the extension $L_0/\mathbb{Q}_p$, where $\varphi$ is Euler's totient function, and put $q=p^N$. (Note that $q \equiv 1 \bmod 2r^2h$ for our choice of $N$.) Then we have the following lemma.

\begin{lem}\label{lem t-exp of F}
For all positive integers $m$ and $k$, the rational function $\tau_T^\ast F_{mq^k}$ has the $t$-expansion
\begin{align}\label{t-exp of F}
(\tau_T^\ast F_{mq^k})(t)=\frac{\widehat{\sigma_{m\tilde{T}}}([mq^k]^{\ast}(t))}{\widehat{\sigma_{\tilde{T}}}(t)^{m^2q^{2k}}}.
\end{align}
\end{lem}

\begin{proof}
Fix an embedding $K \subset \overline{K} \hookrightarrow \mathbb{C}$. We regard $E$ as an elliptic curve over $\mathbb{C}$ and take the period lattice $\Lambda$ corresponding to $E$ and the uniformization $\Phi:\mathbb{C}/\Lambda \overset{\sim}{\rightarrow} E(\mathbb{C})$ so that $z \mapsto (\wp(z; \Lambda), \wp'(z; \Lambda))$, where $\wp(z; \Lambda) \coloneqq -\zeta'(z; \Lambda)$ is the Weierstrass $\wp$-function. Let $\sigma(z)\coloneqq \sigma(z;\Lambda)$ be the Weierstrass sigma function associated to the lattice $\Lambda$. Then, by \cite[Lemma 7]{Si05a}, we have
\begin{align}\label{Lem7}
F_n(\Phi(z))=\frac{\sigma(nz)}{\sigma(z)^{n^2}}
\end{align}
for all $z \in \mathbb{C}$ and $n\geq 1$. 

Choose $w \in \Lambda\otimes\mathbb{Q}$ representing $T_{4r}$, and put $z_T=4rw$. Then $\Phi (z_T)=T$, and for every divisor $d \mid 4r$, 
\begin{align*}
\Phi (z_T/d)=(4r/d)T_{4r}=T_d.
\end{align*}
Let $\tilde{T}^{\rm an}$ be the element of $V(E)$ associated with $z_T$
under the inclusion $\Lambda \otimes \mathbb{Q} \hookrightarrow V(E)$.
Then $\sigma^{\rm alg}(\tilde{T}^{\rm an})=\sigma^{\rm alg}(\tilde{T})$ by
Remark \ref{rmk algsigma}(2), since $T^{\rm an}_{4r}=T_{4r}$.
As $g_{\tilde{T}}$ is determined by its divisor and its normalization at
$O$, which depend only on $T_1$ and $\sigma^{\rm alg}(\tilde{T})$, the
algebraic sigma series attached to $\tilde{T}^{\rm an}$ and $\tilde{T}$
agree. The same finite-level argument applies to $m\tilde{T}$, as
$4N_m \mid 4r$ for the order $N_m$ of $mT$. We may therefore apply
Remark \ref{ASSoverC} to the given lift $\tilde{T}$.

By \eqref{Lem7} and Definition \ref{sigma_z}, we obtain
\begin{align}\label{translation F}
(\tau_T^\ast F_n)(\Phi(z))
&=F_n(\Phi(z+z_T))=\frac{\sigma_{nz_T}(nz)}{\sigma_{z_T}(z)^{n^2}}.
\end{align}
We set $n=mq^k$ in (\ref{translation F}).
Since $q^k \equiv 1 \bmod 2r^2$, there exists an integer $c$ such that $q^k=2r^2c+1$. Since $rz_T \in \Lambda$, it follows from \eqref{transWt'} that 
\begin{align*}
\sigma_{mq^kz_T}(mq^kz)&=\sigma_{mz_T+2r^2cmz_T}(mq^kz) \\
&=\psi(2r^2cmz_T)\langle \frac{1}{2}mz_T, 2r^2cmz_T \rangle \sigma_{mz_T}(mq^kz)=\sigma_{mz_T}(mq^kz).
\end{align*}
(For the last equality, note that $2r^2cmz_T \in 2\Lambda$, so
$\psi(2r^2cmz_T)=1$, and that
$\langle \frac{1}{2}mz_T, 2r^2cmz_T \rangle =1$ since both entries are
multiples of $z_T$.)
Hence it follows from (\ref{translation F}) that
\begin{align*}
(\tau_T^\ast F_{mq^k})(\Phi(z))
=\frac{\sigma_{mz_T}(mq^kz)}{\sigma_{z_T}(z)^{m^2q^{2k}}}.
\end{align*}
Substituting $z=\lambda(t)$ and using Remark \ref{ASSoverC}, we obtain (\ref{t-exp of F}). Since we may view (\ref{t-exp of F}) as an equality in the fraction field of $K(E[4r^2])[[t]]$, this completes the proof of the lemma.
\end{proof}

We put $Q \coloneqq P-T \in \hat{E}(K)$ and take an element $t_Q\in \mathfrak{m}_K$ corresponding to $Q$.

Since $(\tau_T^\ast F_{q^k})^{-1}$ has the divisor
\begin{align*}
{\rm div}((\tau_T^\ast F_{q^k})^{-1})=q^{2k}(-T)-\sum_{q^{k}R=-T}(R),
\end{align*}
its poles are the points $R$ with $q^kR=-T$. No such $R$ lies on the
residue disc at $O$: if $R$ reduces to the origin, then so does $q^kR$,
whereas the reduction of $-T$ is nonzero. Hence
$(\tau_T^\ast F_{q^k})^{-1}$ has no pole on the residue disc at $O$, and
by Lemma \ref{pWPT}(1) its $t$-expansion is integral up to multiplication
by a constant.
Therefore, for sufficiently large $k_0 \in \mathbb{Z}$ (so that $[q^{k_0}]^{\ast}(t_Q)$ satisfies the assumption of Lemma \ref{keylem}(2)),
\begin{align*}
s_{k_0}(t) \coloneqq \widehat{\sigma_{\tilde{T}}}(t)^{q^{2k_0}}=\widehat{\sigma_{\tilde{T}}}([q^{k_0}]^{\ast}(t)) \cdot (\tau_T^\ast F_{q^{k_0}})^{-1} \in K(E[4r^2])[[t]]
\end{align*}
is convergent at $t=t_Q$. Since $F_{q^{k_0}}(P) \equiv F_1(P) =1 \bmod \mathfrak{m}_K$ by our choice of $N$ and $|\widehat{\sigma_{\tilde{T}}}([q^{k_0}]^{\ast}(t_Q))-\sigma^{\rm alg}(\tilde{T})|<1$ by Lemma \ref{keylem}(2), we see that $|s_{k_0}(t_Q)-\sigma^{\rm alg}(\tilde{T})|<1$ and hence $s_{k_0}(t_Q)^{q^{2k-2k_0}} \to \omega(\sigma^{\rm alg}(\tilde{T}))$ as $k \to \infty$. Since $[mq^k]^{\ast}(t_Q) \to 0$, the numerator of (\ref{t-exp of F}) tends to $\sigma^{\rm alg}(m\tilde{T})$ by Lemma \ref{keylem}(2) if $mT \neq O$, and directly from the definition if $mT=O$. Thus, substituting $t=t_Q$ into (\ref{t-exp of F}), we obtain
\begin{align*}
F_{mq^k}(P)
=\frac{\widehat{\sigma_{m\tilde{T}}}([mq^k]^{\ast}(t_Q))}{s_{k_0}(t_Q)^{m^2q^{2k-2k_0}}}
\to \frac{\widehat{\sigma_{m\tilde{T}}}(0)}{\omega(s_{k_0}(t_Q))^{m^2}}= \frac{\sigma^{\rm alg}(m\widetilde{T})}{\omega(\sigma^{\rm alg}(\widetilde{T}))^{m^2}}
\end{align*}
as $k \to \infty$. Since each $F_{mq^k}(P)$ lies in $K$ and the right-hand side is either zero or a $p$-adic unit, the limit belongs to $\mathcal{O}_K$. This completes the proof of the theorem.
\end{proof}

\begin{proof}[Proof of Theorem \ref{main1}]
If $p \nmid r$, we have already proved the assertion in Theorem \ref{main1'}. If $r=p^j$ for some positive integer $j$, it follows from \cite[Theorem 6.1]{St16} that $\lim_{k \to \infty}v_p(F_{mp^k}(P))=+\infty$.

We consider the third case, that is, $j \geq 1$ and $r' \geq 2$. Since the reduction of $P'=p^jP \bmod \mathfrak{m}_K$ has order $r' \geq 2$ and $p \nmid r'$, Theorem \ref{main1'} implies that there exists a power $q=p^N$ such that the limit
\begin{align*}
\lim_{k\to \infty} F_{mq^k} (P')=\frac{\sigma^{\rm alg}(m\widetilde{T})}{\omega(\sigma^{\rm alg}(\widetilde{T}))^{m^2}}
\end{align*}
exists for any positive integer $m$. If we put $m'=mp^{(N-1)j}$, it follows from \cite[Appendix I, Proposition 2]{MT91} that
\begin{align*}
F_{m q^{k+j}} (P)=F_{m'q^k} (P')F_{p^j}(P)^{(m'q^k)^2}.
\end{align*}
Hence we have
\begin{align*}
\lim_{k\to \infty} F_{mq^{k+j}} (P)=\frac{\sigma^{\rm alg}(m'\widetilde{T})}{\omega(\sigma^{\rm alg}(\widetilde{T}))^{m'^2}} \cdot \omega (F_{p^j}(P))^{m'^2}.
\end{align*}
(We also note that $F_{p^j}(P)$ is a $p$-adic unit in $K$ by \cite[Theorem 6.1]{St16}.)
%
%
\end{proof}


\begin{rmk}\label{depends only on T}
Since the right-hand side of Theorem \ref{main1'} is independent of the choice of $\widetilde{T}$ as noted in the proof of  Lemma \ref{lem t-exp of F}, so is that of Theorem \ref{main1}.
\end{rmk}

\section{Application to elliptic divisibility sequences}

In this section, we introduce elliptic divisibility sequences, relate them
to division polynomials (Proposition \ref{prop18}), and derive from
Theorem \ref{main1} explicit $p$-adic limit formulas for nonsingular
elliptic divisibility sequences (Corollary \ref{cor EDS}); in particular,
a conjecture of Silverman (Conjecture \ref{silverman conj}) holds at every
prime $p \geq 5$ of good reduction.
The standard material in Subsection 4.1 is recalled mainly from
\cite[Sections 8--10]{Si05a}; see also \cite{Ay92}, \cite{Ay93},
\cite{Sh00} and  \cite{Wa48}.

\subsection{Elliptic divisibility sequences and division polynomials}

\begin{dfn}
Let $W=(W_n)_{n \geq 0}$ be a sequence of rational integers.
\begin{enumerate}
\item If $W$ satisfies the condition
\begin{align*}
n \mid m \Rightarrow W_n \mid W_m,
\end{align*}
$W$ is called a {\it divisibility sequence}.
\item If a divisibility sequence $W$ satisfies the recurrence relation
\begin{align*}
W_{m+n}W_{m-n}=W_{m+1}W_{m-1}W_n^2-W_{n+1}W_{n-1}W_m^2
\end{align*}
for all $1 \leq n \leq m$, $W$ is called an {\it elliptic divisibility sequence} (abbreviated EDS).
\item We call an elliptic divisibility sequence $W$ {\it proper} if $W_0=0$, $W_1=1$, and $W_2W_3 \neq 0$.
\end{enumerate}
\end{dfn}

Ward \cite{Wa48} proved that every proper elliptic divisibility sequence $W$ determines a cubic curve $E_W$ over $\Q$ together with a point $P_W \in E_W(\Q)$.
We call a proper elliptic divisibility sequence $W$ {\it nonsingular} if
$E_W$ is nonsingular or, equivalently, if Ward's discriminant
${\rm Disc}(W)$ is nonzero.
We refer to Ward \cite{Wa48} and to the appendix of Silverman--Stephens \cite{SS06} for the explicit formulas. 
For our purposes, the essential point is the following
division-polynomial representation.



\begin{prop}[{\cite[Proposition 18]{Si05a}}]\label{prop18}
Let $W=(W_n)_{n \geq 0}$ be a nonsingular EDS and $(E_W, P_W)$ the pair associated with $W$. Fix a minimal Weierstrass equation for $E_W$. Then there exists a constant $\gamma \in \mathbb{Q}^{\times}$ such that for every integer $n \geq 1$,
\begin{align*}
W_n=\gamma^{n^2-1}F_n(P_W).
\end{align*}
Moreover, the denominator of $\gamma$ is divisible only by primes of bad reduction for $E_W$ at which $P_W \bmod p$ is the singular point of the reduced cubic curve.
\end{prop}

\subsection{Explicit $p$-adic limits for elliptic divisibility sequences}

Silverman proposed the following conjecture, motivated by results of Silverman and Stephens \cite{SS06}.

\begin{conj}[{\cite[Conjecture 19]{Si05a}}]\label{silverman conj}
Let $W=(W_n)_{n \geq 0}$ be an EDS and let $p$ be a prime. Then there exists a positive integer $N$ such that for every positive integer $m$, the limit 
\begin{align}\label{lim W}
\lim_{k \to \infty} W_{mp^{kN}}
\end{align}
exists in $\mathbb{Z}_p$ and is algebraic over $\mathbb{Q}$.
\end{conj}

Silverman \cite[Theorem 20]{Si05a} showed that Conjecture \ref{silverman conj} is true for EDSs that are {\it general} in the sense of \cite{Si05a} (that is, for proper nonsingular EDSs whose associated point has infinite order) and for all primes outside an explicitly described exceptional set.
However, as he noted in the addendum \cite{Si05b}, Ayad \cite{Ay93} essentially proved the convergence of the sequence in \eqref{lim W} for all but finitely many primes by different methods.
(Although the addendum \cite{Si05b} did not prove the algebraicity of the limit \eqref{lim W}, it follows from Lemma \ref{pWPT} and the relation
\begin{align*}
W_{mq^{k+1}}=\gamma^{q^2-1}F_q(mq^kP_W)W_{mq^k}^{q^2}
\end{align*}
unless the prime-to-$p$ part of the reduction order of $P_W$ divides $m$.)

As a consequence of Theorem \ref{main1}, we obtain explicit $p$-adic limits for EDSs at every prime $p \geq 5$ of good reduction. This proves Conjecture \ref{silverman conj} for every nonsingular elliptic divisibility sequence at every prime $p \geq 5$ of good reduction; for ordinary reduction, this recovers the corresponding case of \cite[Theorem 20]{Si05a}.

\begin{cor}\label{cor EDS}
Let $W=(W_n)_{n \geq 0}$ be a nonsingular EDS and $(E_W, P_W)$ the pair associated with $W$. 
Let $p \geq 5$ be a prime of good reduction for $E_W$.
Fix a Weierstrass equation of the form \eqref{WeiEq} for $E_W$ that is
defined over $\Z_{(p)}$ and smooth at $p$.
Choose  a nonzero $\gamma \in \mathbb{Z}_{(p)}$ such that 
\[
W_n=\gamma^{n^2-1}F_n(P_W) \qquad (n\geq1).
\]
\begin{enumerate}
\item If $P_W \in \widehat{E_W}(\Q_p)$, then for every positive integer $m$ we have
\begin{align*}
\lim_{k \to \infty} W_{mp^{k}}=0
\end{align*}
with respect to the $p$-adic topology.
\item Suppose that $P_W \not\in \widehat{E_W}(\Q_p)$, and let $q=p^N$ as in Theorem \ref{main1}. Put
\begin{align*}
L_m(E_W, P_W) \coloneqq \lim_{k \to \infty} F_{mq^k}(P_W).
\end{align*}
Then for every positive integer $m$ we have
\begin{align*}
\lim_{k \to \infty} W_{mq^{k}}=
\begin{cases}
0, & \text{if $v_p(\gamma)>0$}, \\
\gamma^{-1}\omega(\gamma)^{m^2} L_m(E_W, P_W), & \text{if $v_p(\gamma)=0$}.
\end{cases}
\end{align*}
In particular, the limit is algebraic. When $v_p(\gamma)=0$, it vanishes if and only if the prime-to-$p$ part of the order of the reduction of $P_W$ divides $m$.
\end{enumerate}
\end{cor}

\begin{proof}
The chosen equation is related to a global minimal equation for $E_W$ by a
change of variables $(x,y) \mapsto (u^2x+r,\, u^3y+sx+t)$ with
$u \in \Q^{\times}$, under which $F_n$ is multiplied by $u^{n^2-1}$. Since
both equations have unit discriminant at $p$, we have $v_p(u)=0$, and the
constant of Proposition~\ref{prop18} yields a constant
$\gamma \in \Z_{(p)}$ as above.

\noindent\textup{(1)} We write $t(P)=-2x/y$ for $P=(x, y) \in \widehat{E_W}(\Q_p)$. From $(6.5)$ in \cite{St16}, we have
\begin{align*}
v_p(F_n(P_W))=-n^2 v_p(t(P_W))+ v_p(t(nP_W)).
\end{align*}
Since $W_n=\gamma^{n^2-1}F_n(P_W)$ for every integer $n \geq 1$, we obtain
\begin{align}\label{val W_n}
v_p(W_n)=( v_p(\gamma)-v_p(t(P_W))) n^2+v_p(t(nP_W))-v_p(\gamma).
\end{align}
Since $[2]^*(t)=2t+O(t^2)$, we have $v_p(t(2P_W))=v_p(t(P_W))$. As $W_2$ is a nonzero integer, the case $n=2$ of \eqref{val W_n} gives $0\leq v_p(W_2)=3\bigl(v_p(\gamma)-v_p(t(P_W))\bigr)$, and hence $v_p(\gamma)\geq v_p(t(P_W))$. Taking $n=mp^k$ in \eqref{val W_n}, we see from $v_p(t(mp^kP_W)) \to \infty$ that its right-hand side tends to $\infty$.

\medskip
\noindent\textup{(2)} If $\gamma$ is a $p$-adic unit, then
\begin{align*}
\gamma^{m^2q^{2k}-1} \to \gamma^{-1}\omega(\gamma)^{m^2} \quad (k \to \infty)
\end{align*}
whereas if $v_p(\gamma)>0$, the same factor tends to zero. Hence the assertion follows from
\begin{align*}
W_{mq^{k}}=\gamma^{m^2q^{2k}-1}F_{mq^k}(P_W)
\end{align*}
and Theorem \ref{main1} together with Corollary \ref{cor alg}.
\end{proof}

\vspace{10pt}

\noindent
Institute of Mathematics for Industry, Kyushu University,\\
744, Motooka, Nishi-ku, Fukuoka, 819-0395, Japan,\\
E-mail address: \textbf{yu.katagiri.s3@gmail.com}\\

\noindent
Faculty of Mathematics, Kyushu University,\\
744, Motooka, Nishi-ku, Fukuoka, 819-0395, Japan,\\
E-mail address: \textbf{kobayashi@math.kyushu-u.ac.jp}

\end{document}